\documentclass[11pt]{amsart}

\usepackage{times}
\usepackage{geometry}
\usepackage{amssymb}
\usepackage{latexsym, amsmath, amscd,amsthm}
\usepackage{graphicx}
\usepackage[percent]{overpic}
\usepackage{pdfsync}
\usepackage{units}
\usepackage{hyperref}
\usepackage{diagrams}
\usepackage{paralist}
\newtheorem{theorem}{Theorem}

\newtheorem{proposition}[theorem]{Proposition}
\newtheorem{corollary}[theorem]{Corollary}

\theoremstyle{definition}
\newtheorem{definition}[theorem]{Definition}

\newtheorem{remark}[theorem]{Remark}

\def\thm#1{Theorem~\ref{thm:#1}}

\def\rmark#1{Remark~\ref{rmark:#1}}
\newcommand{\R}{\mathbb{R}}
\newcommand{\Z}{\mathbb{Z}}
\newcommand{\bdry}{\partial}
\newcommand{\cross}{\times}
\def\co{\colon\!}
\newcommand{\transverse}{\pitchfork}
\newcommand{\SO}{\operatorname{SO}}
\newcommand{\Or}{\operatorname{O}}
\DeclareMathOperator{\Sim}{Sim}
\DeclareMathOperator{\rank}{rank}
\DeclareMathOperator{\id}{id}
\DeclareMathOperator{\Vol}{Vol}
\DeclareMathOperator{\CM}{CM}
\newcommand{\pij}{\pi_{ij}}
\newcommand{\rijl}{r_{ijl}}
\newcommand\norm[1]{\left| #1 \right|}
\newcommand{\p}{\protect\overrightarrow{\mathbf{p}} }
\newcommand{\q}{\protect\overrightarrow{\mathbf{q}} }
\newcommand{\x}{\protect\overrightarrow{\mathbf{x}} }
\newcommand{\bfe}{\mathbf{e}}
\newcommand{\bfp}{\mathbf{p}}
\newcommand{\bfq}{\mathbf{q}}
\newcommand{\bfr}{\mathbf{r}}
\newcommand{\bfv}{\mathbf{v}}
\newcommand{\bfx}{\mathbf{x}}
\newcommand{\bfzero}{\mathbf{0}}
\newcommand{\cnm}{C_n[M]}
\newcommand{\cnmo}{C_n(M)}
\newcommand{\cnr}{C_n[\R^k]}
\newcommand{\cnro}{C_n(\R^k)}
\newcommand{\cnf}{C_n[f]}
\def\mapright#1#2{
\mathop{\longrightarrow}\limits^{\scriptstyle#1}_{\scriptstyle#2}}

\begin{document}
\title{Families of similar simplices inscribed in most smoothly embedded spheres.}
\author{Jason Cantarella}
\address{Mathematics Department, University of Georgia, Athens GA 30602}
\email{jason.cantarella@uga.edu}
\author{Elizabeth Denne}
\address{Mathematics Department, Washington \& Lee University, Lexington VA 24450}
\email{dennee@wlu.edu}
\author{John McCleary}
\address{Mathematics \& Statistics Department, Vassar College, Poughkeepsie NY 12604}
\email{mccleary@vassar.edu}

\makeatletter								% This is needed until the AMS updates their amsart style.
\@namedef{subjclassname@2020}{%
  \textup{2020} Mathematics Subject Classification}
\makeatother


\begin{abstract} 

Let $\Delta$ denote a non-degenerate $k$-simplex in $\R^k$. The
set $\Sim(\Delta)$ of simplices in $\R^k$ similar to $\Delta$ is diffeomorphic to $\Or(k)\times [0,\infty)\times \R^k$, where the factor 
in $\Or(k)$ is a matrix called the {\it pose}.
Among $(k-1)$-spheres smoothly embedded in $\R^k$ and isotopic to the identity, there is a dense family of spheres, for which 
the subset of $\Sim(\Delta)$ of simplices inscribed in each embedded sphere contains a similar simplex of every pose $U\in \Or(k)$. 
Further, the intersection of $\Sim(\Delta)$ with the configuration space of $k+1$ distinct points on an embedded sphere is a 
manifold whose top homology class maps to the top class in $\Or(k)$ via the pose map. This gives a high dimensional 
generalization of classical results on inscribing families of triangles in plane curves.  We use techniques established 
in our previous paper on the square-peg problem where we viewed inscribed simplices in spheres as transverse intersections 
of submanifolds of compactified configuration spaces.
\end{abstract}

\date{\today}
\maketitle

%%%%%%%%%%%%%%%%%%%%%%%%%%%%%%%%%%%%%%%
%%%%%%%%%%%%%%%%%%%%%%%%%%%%%%%%%%%%%%%

\section{Introduction}\label{sect:intro} 
There is a general type of problem of finding special geometric configurations on families of manifolds. Quite often such problems seek to find some kind of polyhedron or polytope inscribed in a circle or sphere embedded in space. Our paper seeks to find constructible non-degenerate simplices inscribed in spheres smoothly embedded in $\R^k$.   Specifically, for a given non-degenerate $k$-simplex $\Delta$, we show that among all smoothly embedded $(k-1)$ spheres in $\R^k$ isotopic to the identity through a differentiable isotopy in $\R^k$, there is a dense family of spheres such that each sphere has an inscribed simplex similar to $\Delta$ corresponding to each $U\in \Or(k)$.  An example of an  embedded 2-sphere in $\R^3$ with an inscribed equilateral tetrahedra is found in Figure~\ref{fig:tetra}.
We in fact prove more: If we let $\Sim(\Delta)$ denote the configuration space of simplices in $\R^k$ similar to a non-degenerate $k$-simplex $\Delta$, then in $\Sim(\Delta)$, the top homology class of the inscribed simplices similar to $\Delta$ on an embedded sphere  maps to the top class in $\Or(k)$. 

In 1969, M.L. Gromov~\cite{MR244907} showed that every $C^1$-smooth embedding of a $(k-1)$-sphere in $\R^k$ %which bounds a ball on both sides 
contains an inscribed simplex similar to $\Delta$ for each pose $U\in \Or(k)$.  Our results go further than Gromov. We show that the inscribed simplices in embedded spheres form a manifold cobordant to $\Or(k)$ using configuration spaces and transversality as in~\cite{Slq}.

Another closely related result is the 2009 work of P.V.M.~Blagojevi\'c and G.~Ziegler~\cite{MR2515779}. They prove that for every injective continuous maps of $f\co S^2 \rightarrow \R^3$, there are four distinct points in the image of $f$ with the property that two opposite edges have the same length and the other four edges are also of equal length.  While this result holds for injective maps, our theorem recovers an even more general result for generic smooth embeddings. In addition, B. Matschke~\cite{phd-Matschke} proved that any smoothly embedded compact surface $S$ inscribes a particular type of tetrahedron.

\begin{figure}
\hfill
\includegraphics[height=1.65in]{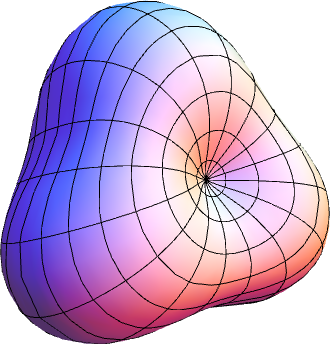}
\hfill
\includegraphics[height=1.65in]{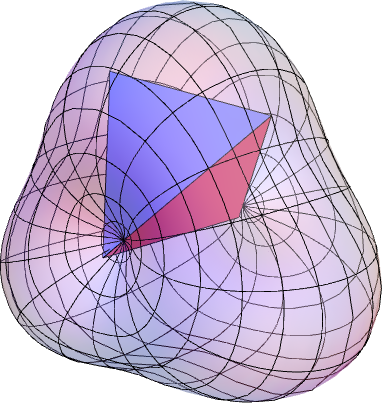}
\hfill
\includegraphics[height=1.65in]{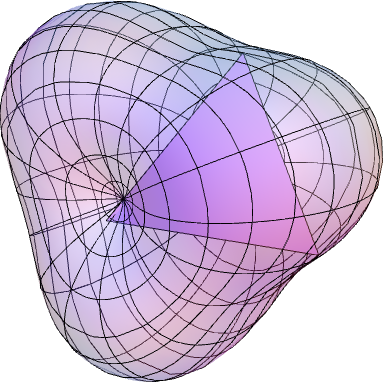}
\hfill
\hphantom{.}
\caption{On the left, we see an irregular embedding of $S^2$ in $\R^3$ described in spherical coordinates as a graph over the unit sphere by the function $r(\phi,\theta) = 1 + \sin^3\phi \sin 3\theta/5 - |\cos^7\phi|$. The center and right images show different views of a single regular tetrahedron inscribed in this surface with edgelengths close to $1.15$. If this embedding of $S^2$ is transverse to the submanifold of regular tetrahedra, this tetrahedron is a member of the family of inscribed regular tetrahedra predicted by Theorem~\ref{thm:main-theorem}. This tetrahedron was found by computer search. Its vertices have spherical $(\phi,\theta)$ coordinates $(0.224399, 0.224399), (1.5708, 3.36599), (1.5708, 2.0196), (2.91719, 0.224399)$.
\label{fig:tetra}}
\end{figure}

All of these theorems generalize the classical results on inscribing families of triangles in planar and spatial curves. In 1978, M.D.~Meyerson~\cite{MR600575} proved that  for a fixed arbitrary triangle, every simple closed curve in the plane contains the vertices of a triangle similar to the given one. He also proved that for every simple closed curve $\gamma$ in the plane, then for all, except perhaps two, points $x$ on $\gamma$, we can find points $y$ and $z$ in $\gamma$ such that $xyz$ is an equilateral triangle. In 1992, M.J.~Nielsen~\cite{MR1181760} proved that for any triangle $T$ and any simple closed curve $\gamma$ in the plane, there are infinitely many triangles similar to $T$ inscribed in $\gamma$. In fact, he proves that the set of vertices corresponding to the vertex of smallest angle in $T$ is dense in $\gamma$. In 2011, Matschke \cite{phd-Matschke} proved a more general result. He looked at $n$-gons whose edgelengths are in a prescribed ratio. He showed such polygons are inscribed in generic $C^\infty$-smooth embeddings of $S^1$ into a (complete) Riemannian manifold, and moreover there is a 1-parameter family of such polygons. He proves that there are an odd number of loops of such polygons that wind an odd number of times around the embedded $S^1$. Our main theorem and Matschke's results are very similar for triangles inscribed in embedded curves in the plane, except that we prove (Corollary~\ref{cor2}) that the degree of the map is 1.  In 2021, A. Gupta and S. Rubinstein-Salzedo \cite{gupta2021inscribed} generalized Nielsen's result to any Jordan curve embedded in $\R^n$ for a restricted set of triangles dependent on certain geometric conditions. They then get a wider class of inscribed triangles by adding in regularity conditions. Namely, they show that if the curve is differentiable at a point, then any triangle can be inscribed at that point.  %In addition, if the curve is strongly locally monotone at a point (an idea from W.~Stromquist~\cite{MR1045781}), then the point is a vertex of an inscribed equilateral triangle in the curve.

We pause to note that many of the inscribed triangle results were inspired by the square-peg problem: finding four points on any Jordan curve (a simple closed curve in the plane) which are the vertices of a square. This question was posed by O.~Toeplitz in 1911 \cite{Toeplitz}, and progress on this problem has chiefly been extension of the regularity class of simple closed curves for which the square can be found. The interested reader can find numerous articles \cite{FTCWC, MR1133201, MR3184501, Pak-Discrete-Poly-Geom,   MR3731730}  summarizing the problem, and describing the classes of curves for which the square-peg problem has been proved.  There have also been many  papers  \cite{MR3810027, MR4061975,  MR4298749, hugelmeyer2018smooth, MR4298748, MR2038265, matschke2020quadrilaterals, MR4142923, MR2823976} examining  quadrilaterals and polygons inscribed in curves and, more recently, making progress towards solving the rectangular-peg problem (finding rectangles of any aspect ratio inscribed in Jordan curves).  

There are yet more results about inscribed and circumscribed polyhedra in spheres. For example, S.~Kakutani's theorem \cite{MR7267} that a compact convex body in $\R^3$ has a circumscribed cube,  that is, a cube each of whose faces touch the convex body. As Matschke~\cite{MR3184501} notes, most smooth embeddings $S^{k-1}\hookrightarrow \R^k$ do not inscribe a $k$-cube for $k\geq 3$. (Intuitively, the number of equations to fulfill is larger than the degrees of freedom.) Instead asking whether crosspolytopes\footnote{The regular $k$-dimensional polytope is the convex hull of $\{\pm e_i\}$ where $e_i$ are the standard basis vectors in $\R^k$. } are inscribed in such an embedding might be a better generalization of inscribing a square.  V.V.~Makeev~\cite{phd=Makeev} proved the $k=3$ case, and R.N.~Karasev~\cite{karasev2009inscribing} generalized the proof to arbitrary odd prime powers. 
Later, A.~Akoypan and Karasev \cite{MR3016979} used a limit argument to show that a simple convex  polytope admits an inscribed regular octahedron.   In addition, there is the work of G.~Kuperberg \cite{MR1703205} and Makeev \cite{MR2307358} on inscribed and circumscribed polyhedra in convex bodies and spheres.

The inscribed simplex problem can be framed in terms of compactified configuration spaces. We give a short overview of the basics about these spaces in Section~\ref{section:config}.  We consider the compactified configuration space $C_{k+1}[\R^k]$ of ordered $(k+1)$-tuples of points in $\R^k$ as a manifold-with-boundary (and corners). The existence of inscribed simplices in embedded spheres can be viewed as finding the intersection of two submanifolds of $C_{k+1}[\R^k]$. The first is the submanifold of $(k+1)$-tuples 
of points on a $C^\infty$-smooth embedding $\gamma \co S^{k-1}\hookrightarrow \R^k$ of a $(k-1)$-sphere in $\R^k$; the second is the submanifold of simplices in $\R^k$ which are similar to a given simplex $\Delta$, denoted by $\Sim(\Delta)$.    

Theorems about intersections of manifolds are often proved by transversality arguments. There are many examples in the literature~\cite{Guillemin:2010ti,MR314068,MR61823}. In~\cite{Slq}, we provide a framework for this kind of argument adapted to configuration spaces, which we will use again in this paper. Section~\ref{sect:main-method} gives a description of this method and we give a very quick overview here. We consider a smooth embedding $\gamma\co S^l \hookrightarrow \R^k$ of $S^l$ in $\R^k$ Then, in the compactified configuration space $C_n[\R^k]$ of $n$ points in $\R^k$, we consider $C_n[\gamma(S^l)]$, the compactified configuration space of $n$ points on $\gamma(S^l)$; and $Z$ the subspace of tuples of $n$ points in $\R^k$ that satisfy the conditions of a special configuration. Now, suppose there is a different, well known, smooth embedding $i\co S^l\hookrightarrow \R^k$ of $S^l$ in $\R^k$, and assume that the configuration space $C_n[i(S^l)]$ is transverse to $Z$ in $C_n[\R^k]$. We can use Haefliger's theorem \cite{Haefliger:1961wr} to find a differentiable isotopy between $i(S^l)$ and $\gamma(S^l)$ (the differentiable isotopy may need to go through $\R^K$ for $K\geq k$). The key idea is that we ought to be able to vary $C_n[i(S^l)]$ to $C_n[\gamma(S^l)]$ while maintaining the transversality of the intersection with $Z$. To do so, we need to make several assumptions about $Z$ and $C_n[\gamma(S^l)]$ (see Section~\ref{sect:main-method}), and to use technical tools like multijet transversality. In the end, we are able to deduce that there is, for all $m$, a $C^m$-dense\footnote{The density is with respect to the Whitney $C^\infty$-topology, described in Section~\ref{sect:main-method}.} set of smooth embeddings $\gamma'\co S^l\hookrightarrow \R^k$, such that the corresponding embeddings $C_n[\gamma']$ on configuration spaces are $C^0$-close to $C_n[\gamma]$, and that $C_n[\gamma'(S^l)]$ is transverse to $Z$, and moreover $C_n[i(S^l)] \cap Z$ and $C_n[\gamma'(S^l)] \cap Z$ represent the same homology class in $Z$.  We apply this method to the inscribed simplices problem in the final two sections of the paper.

In Section~\ref{sect:simplices}, we first describe a non-degenerate simplex in terms of the distances between distinct vertices. We then use the Cayley-Menger determinant  (see Theorem~\ref{thm:CM}) to give a description of when it is possible to construct a simplex from a set of distances. Secondly, we change our perspective and view a simplex $\Delta$ in $\R^k$ as an ordered $(k+1)$-tuple of distinct points in the configuration space $C_{k+1}(\R^k)$. We then define $\Sim(\Delta)$ to be the space of simplices similar to $\Delta$ (through a translation, rotation, or nonzero scaling of $\R^k$).  In Theorems~\ref{thm:SimD} and \ref{thm:SimD-bdry} we prove that $\Sim(\Delta)$ is a submanifold of $C_{k+1}[\R^k]$ diffeomorphic to $\Or(k)\times [0,\infty)\times \R^k$.

In Section~\ref{sect:inscribed} we apply our method from Section~\ref{sect:main-method}. Given a $C^\infty$-smooth embedding $\gamma\co S^{k-1}\hookrightarrow \R^k$ of $S^{k-1}$ in $\R^k$, we let $C_{k+1}[\gamma(S^{k-1})]$ be the submanifold of $C_{k+1}[\R^k]$ corresponding to $(k+1)$-tuples of points on the embedded sphere $\gamma(S^{k-1})$. In Proposition~\ref{prop:bdry-disjoint}, we show that $\Sim(\Delta)$ and $C_{k+1}[\gamma(S^{k-1}]$ are boundary disjoint in $C_{k+1}[\R^k]$. We then restrict our attention to the standard embedding of $S^{k-1}$ in $\R^k$, with corresponding configuration space $C_{k+1}[S^{k-1}]$. In Proposition~\ref{prop:transverse}, we show that $\Sim(\Delta)$ intersects $C_{k+1}[S^{k-1}]$ transversally and the intersection $\Sim(\Delta)\cap C_{k+1}[S^{k-1}]$ is diffeomorphic to $\Or(k)$.  Since $\Or(k)$ is disconnected, we restrict our attention to $\Sim^+(\Delta)$ which is the submanifold diffeomorphic to $\SO(k)\times [0,\infty) \times \R^k$. We then show in Proposition~\ref{prop:homology}  that 
the top class in $\Sim^+(\Delta)$ corresponds to the top homology class of $\SO(k)$ which also corresponds to the top class of the intersection $\Sim^+(\Delta)\cap C_{k+1}[S^{k-1}]$.   Finally in Theorem~\ref{thm:main-theorem}, we prove there is a dense set of smooth embeddings of $S^{k-1}$ in $\R^k$ which are isotopic to the identity through a differentiable isotopy in $\R^k$, such that the subset of $\Sim^+(\Delta)$ of simplices inscribed in each sphere contains a similar simplex corresponding to each $U\in \SO(k)$.   While we chose to restrict our attention to $\Sim^+(\Delta)$, the same results hold for $\Sim^-(\Delta):=\Sim(\Delta)\setminus \Sim^+(\Delta)$. 
In the special case of a smooth embedding of a circle in the plane, our results show that there are loops of triangles inscribed on an embedded circle $C^0$-close to the given embedding.  (See Corollary~\ref{cor2} for a precise statement.)

%%%%%%%%%%%%%%%%%%%%%%%%%%%%%%%%%%%%%%%%%%%%%%%%%%%
      
\section{Configuration Spaces}
\label{section:config}
The compactified configuration space of $k+1$ points in $\R^k$ is the natural setting for finding inscribed simplices in embedded spheres. In this section we give a very brief overview of compactified configuration spaces. There are many versions of this classical material (see for instance~\cite{MR1259368,MR1258919}), but we follow D. Sinha~\cite{newkey119} as this approach is more appropriate to our work. A discussion similar to the one found below is found in \cite{Slq}.

A reader familiar with configuration spaces may skip much of this section. However we recommend paying attention to the notation we have used for the spaces. Definitions~\ref{def:pijsijk}, \ref{def:config}, \ref{def:function-metric}, \rmark{notation} and \thm{functor} are particularly useful. 

\begin{definition}[\cite{newkey119}]\label{def:openconfig}
Given an $m$-dimensional smooth manifold $M$, let  $M^{\times n}$ denote the $n$-fold product $M$ with itself, and define $\cnmo$ to be the subspace of points  ${\mathbf{p}}=(p_1,\dots, p_n)\in M^{\times n}$ such that $p_j\neq p_k$ if $j\neq k$. Let $\iota$ denote the inclusion map of $\cnmo$ in~$M^{\times n}$. 
\end{definition}

The space $\cnmo$ is an open submanifold of $M^{\times n}$. We next compactify $\cnmo$ to a closed manifold-with-boundary and corners, which we will denote $\cnm$, without changing its homotopy type. The resulting manifold will be homeomorphic to $M^{\times n}$ with an open neighborhood of the fat diagonal
%all subdiagonals 
removed. Recall that the fat diagonal is the subset of $M^{\times n}$ of $n$-tuples for which (at least)
two entries are equal, that is, where
some collection of points comes together at a single point. The construction of $\cnm$ preserves information about the directions and relative rates of approach of each group of collapsing points.

\begin{definition}[\cite{newkey119},\cite{newkey118}]  \label{def:pijsijk}
Given an ordered pair $(i,j)$ of distinct elements from $\{1,\dots,n\}$, let the map $\pij\co C_n(\R^k)\rightarrow S^{k-1}$ send~$\bfp=(\bfp_1,\dots\bfp_n)$ to $\displaystyle\frac{\bfp_i-\bfp_j}{\norm{\bfp_i-\bfp_j}}$, the unit vector in the direction of $\bfp_i-\bfp_j$.  Let $[0,\infty]$ be the one-point compactification of $[0,\infty)$.  Given an ordered triple $(i,j,l)$ of distinct elements in $\{1,\dots,n\}$, let $r_{ijl} \co C_n(\R^k)\rightarrow [0,\infty]$ be the map which sends $\bfp$ to $\displaystyle\frac{\norm{{\bfp_i}-{\bfp_j}} }{\norm{ \bfp_i-{\bfp_l}} }$, the ratio of distances between $\bfp_i$ and $\bfp_j$, and $\bfp_i$ and $\bfp_l$. \end{definition}

 We then compactify $C_n(\R^k)$ as follows:

\begin{definition} [\cite{newkey119}] \label{def:config} 
\begin{compactenum}\item Let $A_n[\R^k]$ be the product $(\R^k)^{n}\times (S^{k-1})^{n(n-1)} \times [0,\infty]^{n(n-1)(n-2)}$.  Define $C_n[\R^k]$ to be the closure of the image of $C_n(\R^k)$ under the map
$$\alpha_n= \iota \times (\pij) \times (s_{ijl}) \co C_n(\R^k)\rightarrow A_n[\R^k].$$
\item  We assume that all manifolds $M$ are smoothly embedded in $\R^k$, which allows us to define the restrictions of the maps $\pij$ and $r_{ijl}$.  Then $\cnmo$ is smoothly embedded in $C_n(\R^k)$ and we define $\cnm$ to be the closure of $\cnmo$ in $M^{n}\times (S^{k-1})^{n(n-1)} \times [0,\infty]^{n(n-1)(n-2)}$. 
We denote the boundary of $\cnm$ by $\bdry\cnm=\cnm\setminus\cnmo$.
\end{compactenum}
\end{definition}

We now summarize some of the important features of this construction, including the fact that $\cnm$ does not depend on the choice of embedding of $M$ in $\R^k$.

\begin{theorem}[\cite{newkey118, newkey119}] \label{thm:config}
\begin{compactenum}
\item[\rm{1.}] $\cnm$ is a manifold-with-boundary and corners with interior $\cnmo$ having the same homotopy type as~$\cnm$. The topological type of $\cnm$ is independent of the embedding of $M$ in $\R^k$, and $\cnm$ is compact when $M$ is.
\item[{\rm 2.}] The inclusion of $\cnmo$ in $M^{\times n}$ extends to a surjective map  from $\cnm$ to $M^{\times n}$ which is a homeomorphism over points in $\cnmo$. 
\end{compactenum}
\end{theorem}

\begin{remark}\label{rmark:notation}
When discussing points in $\cnr$ or $\cnm$, it is easy to become confused. We pause to clarify notation.
\begin{itemize}
\item A point in $\R^k$ is denoted by $\bfx=(x_1, \dots, x_k)$, where each $x_i\in\R$.
\item Points in $(\R^k)^{\times n}$ are also denoted by $\bfx$, where $\bfx = (\bfx_1, \dots, \bfx_n)$ and each $\bfx_i\in\R^k$. (It will be clear from context which is meant.)
\item A point in $\cnr$ or $\cnm$, is denoted $\x$. 
 \end{itemize}
\end{remark}

The space $\cnm$ may be viewed as a polytope with a combinatorial structure based on the different ways groups of points in $M$ can come together. This structure defines a {\em stratification} of $\cnm$ into a collection of closed faces of various dimensions whose intersections are members of the collection. 
Full details can be found in \cite{newkey118, newkey119}. The main structure we will consider is the $(0,1,\dots,k)$-face of $\bdry\cnm$. This is the boundary component where all the points come together at the same time.

Any pair $\bfp$, $\bfq$ of disjoint points in $\R^k$ has a direction $(\bfp-\bfq)/\norm{\bfp-\bfq}$ associated to it, while every triple of disjoint points $\bfp$, $\bfq$, $\bfr$ has a corresponding distance ratio $\norm{\bfp-\bfq}/\norm{\bfp-\bfr}$. One way to think of the coordinates of $\cnm$ is that they extend the definition of these directions and ratios to the boundary.

\begin{theorem} [\cite{newkey118, newkey119}] \label{thm:configgeom} Given a manifold $M \subset \R^k$, then in any configuration of points $\p \in \cnm$ the following holds.
\begin{compactenum}
\item  Each pair of points $\bfp_i$, $\bfp_j$ has associated to it a well-defined unit vector in $\R^k$ giving the direction from $\bfp_i$ to $\bfp_j$. If the pair of points project to the same point $\bfp$ of $M$, this vector lies in $T_\bfp M$. 
\item Each triple of points $\bfp_i$, $\bfp_j$, $\bfp_k$ has associated to it a well-defined scalar in $[0,\infty]$ corresponding to the ratio of the distances $\norm{\bfp_i - \bfp_j}$ and $\norm{\bfp_i - \bfp_k}$. If any pair of $\{\bfp_i, \bfp_j,\bfp_k\}$ projects to the same point in $M$ (or all three do), this ratio is a limiting ratio of distances.
\item The functions $\pi_{ij}$ and $r_{ijl}$ are continuous on all of $\cnm$ and smooth on each face of $\bdry\cnm$. 
\end{compactenum}
\end{theorem}

It turns out that for connected manifolds of dimension at least $2$, the combinatorial structure of the strata of $\cnm$ depends only on the number of points. Regardless of dimension, this construction and division of $\bdry \cnm$ into strata is functorial in the following sense.

\begin{theorem}[\cite{newkey119}]\label{thm:functor}
Suppose $M$ and $N$ are embedded submanifolds of $\R^k$ and $f\co  M\hookrightarrow N$ is an embedding. This induces an embedding of manifolds-with-corners called the evaluation map $C_n[f] \co C_n[M]\hookrightarrow C_n[N]$ that respects the stratifications.  This map is defined by choosing the ambient embedding of $M$ in $\R^k$ to be the composition of $f$ with
the ambient embedding of $N$.
\end{theorem}

For an embedding $f \colon M \hookrightarrow N$, the image of the induced embedding $C_n[f]\colon C_n[M] \hookrightarrow C_n[N]$ will be denoted by $C_n[f(M)]$.

\begin{corollary}\label{cor:smooth}
Let $f\co\R^k\rightarrow \R^k$ be a smooth diffeomorphism. Then the induced map of configuration spaces $\cnf\co
\cnr\rightarrow\cnr$ is also a smooth diffeomorphism (on each face of $\cnr$).
\end{corollary}

\begin{proof} This is an immediate corollary of the previous theorem. %\cite{MR2099074}, Theorem 4.7.
\end{proof}

Finally, we will need a metric on the set of evaluation maps $C_n[f] \co C_n[M]\hookrightarrow C_n[N]$. The definition of compactified configuration spaces allows us to view $C_n[N] \subset (\R^k)^{n} \times (S^{k-1})^{n(n-1)} \times [0,\infty]^{n(n-1)(n-2)}$ as a metric space with the sup norm.  If we define the mapping $\mbox{\rm pr}_i$ to be the projection onto the $i$th space of the product, then this naturally leads to a metric on the set of continuous functions $C^0(C_n[M], C_n[N])$.

\begin{definition} With the above assumptions, the metric on the set  $C^0(C_n[M], C_n[N])$ is given by 
$$\|F - G\|_0 = \sup_{\p \in C_n[M]} \{\| \mbox{\rm pr}_i(F(\p)) - \mbox{\rm pr}_i(G(\p))\| \mid \mbox{\rm for all\ }i\}.$$
Thus, given a embeddings $f, g\co M\hookrightarrow N$ we say that the corresponding maps on configuration spaces $C_n[f],C_n[g]\co C_n[M]\hookrightarrow C_n[N]$ are {\em $C^0$-close} if for all $\epsilon >0$, we have $\|C_n[f]-C_n[g]\|_0<\epsilon$.
\label{def:function-metric}
\end{definition}

%%%%%%%%%%%%%%%%%%%%%%%%%%%%%%%
%%%%%%%%%%%%%%%%%%%%%%%%%%%%%%%%%%%%%%%
\section{Finding special configurations with multijet transversality}
\label{sect:main-method}

In this section we set up a general method for tackling problems where we seek a special configuration of $n$ points on a compact manifold $M$ that is smoothly embedded in $\R^k$ (cf. \cite{Slq}).  In both the square-peg problem and the inscribed simplex problem, $M$ is a sphere (either $S^1$ or $S^{k-1}$). We thus denote the $C^\infty$ smooth embedding by $\gamma\co S^l \hookrightarrow \R^k$, and the corresponding  compactified configuration space of $n$ points on $\gamma(S^l)$ is $C_{n}[\gamma(S^l)]$. We let $Z$ denote the subspace of tuples of $n$ points in $\R^k$ that satisfy the conditions of a special configuration. For example, in \cite{Slq}, $Z$ is the set of all square-like quadrilaterals in $\R^k$. In this paper we are interested in $Z=\Sim(\Delta)$ (see Section~\ref{sect:simplices}), which is the set of all $(k+1)$-simplices in $\R^k$ which are similar to a given non-degenerate simplex $\Delta$. 

The central idea is as follows: suppose there is a different, well understood, smooth embedding of $S^l$ in $\R^k$ (via $i\co S^l\hookrightarrow \R^k$), and assume that the corresponding configuration space $C_n[i(S^l)]$ is transverse to $Z$ in $C_n[\R^k]$. Also assume that $i(S^l)$ is smoothly homotopy equivalent to $\gamma(S^l)$ in $\R^k$. Standard transversality arguments should allow us to vary $C_n[i(S^l)]$ to $C_n[\gamma(S^l)]$ while maintaining the transversality of the intersection with $Z$.    There are various technical obstacles to overcome, all of which are handled in detail in \cite{Slq}.  Here are the key steps.

{\bf Step 1:} It is possible that special configurations on $\gamma(S^l)$ shrink away to the boundary of $C_n[\R^k]$ during the isotopy.  To prevent this, we first prove 
\begin{compactenum} \item[(a)] that $n$-tuples of points in $\R^k$ satisfying the geometric condition, $Z$,  are a submanifold in $\cnro$, and that $\bdry Z\subset \bdry \cnr$;
\item[(b)] that $C_n[\gamma(S^l)]$ and $Z$ are boundary-disjoint.
\end{compactenum} 

For $Z=\Sim(\Delta)$, we prove 1(a) in Theorems~\ref{thm:SimD} and \ref{thm:SimD-bdry}. We prove 1(b) in Proposition~\ref{prop:bdry-disjoint}.

{\bf Step 2:} For a standard embedding $i\co S^l\hookrightarrow \R^k$,  we need to do two things
\begin{compactenum} \item[(a)] Prove that the intersection between  $C_n[i(S^l)]$ and $Z$ is non-empty and transverse (in other words, $C_n[i]\pitchfork Z$). 
\item[(b)] Compute the homology class of the intersection $C_n[i(S^l)] \cap Z$ in~$Z$.
\end{compactenum}

For our inscribed simplex problem, we use the standard embedding $\id \co S^{k-1}\hookrightarrow \R^k$  of $S^{k-1}$ in $\R^k$. We prove 2(a) in Proposition~\ref{prop:transverse} and 2(b) in Proposition~\ref{prop:homology}.

{\bf Step 3:} Note that in order to apply our transversality arguments, we need to be able to perturb $C_n[\gamma(S^l)]$ so the intersection of $Z$ remains transverse. However, there is no guarantee that the perturbed submanifold consists of configurations on a perturbed smooth embedding of $S^l$ in $\R^k$. We deal with this issue by applying the multijet transversality theorem~\cite[Theorem II.4.13]{Golubitsky:1974iu}. This allows us to conclude (see \cite{Slq} Theorem 17) that for any $\epsilon >0$,  there is a  $C^\infty$-open neighborhood of $\gamma$ in which there is, for all $m$, a $C^m$-dense set of smooth embeddings $\gamma'\co S^l \hookrightarrow \R^k$, such that $\|C_n[\gamma'] - C_n[\gamma]\|_0< \epsilon$, and $C_n[\gamma'] \transverse Z$, and for which $\bdry Z$ and $\bdry C_n[\gamma'(S^l)]$ are disjoint in $\bdry \cnr$. 

In order to fully appreciate Step 3, first recall that the metric on the set $C^0(C_n[S^l], C_n[\R^k])$ was given in  Definition~\ref{def:function-metric}. Second, to understand the statement about density, we need to give the topology of the spaces we are working in. In general, for manifolds $M(=S^l)$ and $N(=\R^k)$, the space $C^\infty(M,N)$ has the Whitney $C^\infty$-topology (see for instance \cite{Hirsch}).   The sets of the form $$\mathcal{N}^t(f; (U,\phi), (V,\psi), \delta)$$ give a subbasis for the Whitney $C^t$-topology on $C^t(M,N)$ (where $t$ is finite). 
This subbasis is the subset of functions $g\colon M\to N$ that are smooth, and for coordinate charts $\phi\colon (U' \subset M) \to (U\subset \R^m)$ and $\psi\colon 
(V' \subset N) \to (V\subset \R^k)$ and $K \subset U$ compact with $g(\phi(K)) \subset V'$, then we have, for all $s \leq r$, and all $x \in \phi(K)$, 
$$\| D^s (\psi g \phi^{-1})(x) - D^s(\psi f \phi^{-1})(x)\| < \delta. $$
Here, $D^s F$ for a function $F \colon (U\subset \R^m) \to (V \subset \R^k)$ is the $k$-tuple of the $s$th  homogeneous parts of the Taylor series representations of
the projections of $F$.   Finally, the subspace $C^\infty(M,N)$ has the Whitney $C^\infty$-topology by taking the union of all subbases for all $t\geq 0$. 

Note that Step 3 can be applied immediately to the inscribed simplex problem. 

{\bf Step 4:} We need to deform standard spheres into spheres of interest and then consider what happens on the level of configuration spaces. We know precisely when such a deformation of spheres exists due to a result of A.~Haefliger (which we have stated in a form useful to us). 
\begin{theorem}[\cite{Haefliger:1961wr}] 
Any two differentiable embeddings of $S^l$ in $\R^k$ are homotopic through a differentiable
isotopy in $\R^K \supset \R^k$ when $K > 3(l+1)/2$.
\label{thm:isotopy}
\end{theorem}

We use Theorem~\ref{thm:isotopy} to find a smooth map $E\co S^l \cross I \rightarrow \R^K$ with $E(-,0)=i$ our standard embedding and $E(-,1) = \gamma'$ (where $K$ may be greater than our original $k$).  Recalling that both $C_n[i]$ and $C_n[\gamma']$ are transverse to $Z$ allows us to conclude that using functorality, we get a homotopy $H\co C_n(S^l) \times I\rightarrow C_n(\R^K)$ with  $H(-,0)=C_n[i]$ and $H(-,1)=C_n[\gamma']$. In \cite{Slq}, we then use the Transversality Homotopy Extension Theorem (see \cite{Guillemin:2010ti}), to find a map $H'$ homotopic to $H$ and with $H'(-,0)=H(-,0)$, $H'(-,1)=H(-,1)$, and $H'$ is transverse to  $Z$. This then implies that in $Z$, the intersections $C_n[i(S^l)] \cap Z$ and $C_n[\gamma'(S^l)] \cap Z$  represent the same homology class. In other words we have sketched a proof of the following:

\begin{theorem}[\cite{Slq} Theorem  20] Suppose there are two embeddings $\eta, i:S^l\hookrightarrow \R^k$ of an $l$-sphere in $\R^k$. Assume that $Z$ is a closed topological space contained in $\cnr$ such that $Z \cap \cnro$ is a submanifold of $\cnro$, $\bdry Z\subset \bdry\cnr$, and $\bdry Z$ is disjoint from $\bdry C_n[i(S^l)]$.  Also assume that both $C_n[i]$ and $C_n[\eta]$ are transverse to $Z$. Then in $Z$,  the homology class of $C_n[i(S^l)]\cap Z$ and $C_n[\eta(S^l)]\cap Z$  are equal.
\label{thm:homology}
\end{theorem}

In this paper we will not be making full use of Haefliger's Theorem in Step 4. Instead, we will restrict our attention to smooth embeddings of $S^{k-1}$ in $\R^k$ which are differentiably isotopic to the identity through a differentiable isotopy in $\R^k$. The conclusion of Theorem~\ref{thm:homology} still holds for such embeddings. Putting all of our steps together allows us to conclude the following theorem.

\begin{theorem}[\cite{Slq} Theorem 21] \label{thm:method}
Suppose $\gamma\co S^l\hookrightarrow \R^k$ is a smooth embedding of $S^l$ in $\R^k$, with a corresponding embedding of compactified configuration spaces $C_n[\gamma]\co C_n[S^l]\hookrightarrow C_n[\R^k]$.
Assume that $Z$ is a closed topological space contained in $\cnr$ such that $Z \cap \cnro$ is a submanifold of $\cnro$, and $\bdry Z\subset \bdry\cnr$. Also assume that $C_n[\gamma(S^l)]$ and $Z$ are boundary-disjoint.  Suppose there is a standard embedding $i\co S^l\hookrightarrow \R^k$, such that $C_n[i]\pitchfork Z$ in $\cnr$. 

Then for all $\epsilon >0$, there is a  $C^\infty$-open neighborhood of $\gamma$, in which there is, for all $m$, a $C^m$-dense set of smooth embeddings $\gamma':S^l\hookrightarrow \R^k$, such that $\| C_n[\gamma'] - C_n[\gamma]\|_0< \epsilon$, and $C_n[\gamma'] \transverse Z$, and moreover, $C_n[i(S^l)] \cap Z$ and $C_n[\gamma'(S^l)] \cap Z$ represent the same homology class in $Z$. 
\end{theorem}

Intuitively, this theorem shows that any smooth embedding $\gamma$ of $S^l$ in $\R^k$ has a neighborhood in which there is a dense set of smooth embeddings $\gamma'$ for which $C_n[\gamma'(S^l)]$ is guaranteed to have certain intersections with various submanifolds of $\cnr$ defined by geometric conditions.   Stated in a different way, we have the idea that a dense set of embeddings of $S^l$ always contain certain inscribed configurations of points. 

%%%%%%%%%%%%%%%%%%%%%%
%%%%%%%%%%%%%%%%%%%%%%%%%%

\section{Simplices}
\label{sect:simplices}

\subsection{Non-degenerate simplices}
\begin{definition} By a {\em simplex} $\Delta$ in $\R^k$, we mean a set of $k+1$ distinct points $\{\bfp_0, \dots, \bfp_k\}$ in $\R^k$ in general position. 
\end{definition}
By general position, we mean that no hyperplane in $\R^k$ contains more than $k$ points of $\Delta$.
As a first consequence of this definition, we note that the volume of the simplex $\Delta$ is nonzero. This means that for us, {\bf a simplex is  non-degenerate}. In addition, to each simplex $\Delta$, we can associate several sets: 
\begin{compactitem}
\item the nonzero distances $\{d_{ij}(\Delta)\} =\{\|\bfp_i-\bfp_j \|\}$, and observe that $d_{ij}=d_{ji}$;
\item the unit vectors $\{\pij (\Delta)\} = \{\frac{\bfp_i-\bfp_j}{\|\bfp_i-\bfp_j\|}\}$;
\item the ratios $\{\rijl (\Delta)\}= \{\frac{\|\bfp_i-\bfp_j\|}{\|\bfp_i-\bfp_l\|}\}= \{\frac{d_{ij}}{d_{il}}\}$. 
\end{compactitem}
Here, we assume that $i\neq j$ (and $j\neq l, i\neq l$), so the set of $k(k+1)/2$ distances consists of nonzero values, and the set of ratios has values in $(0,\infty)$.  

A natural question to ask, is given a set of nonzero distances $\mathcal{D}=\{d_{ij}\}$, is there a simplex that can be constructed with those distances?  The theory of distance geometry allows us to decide which sets of distances $\mathcal{D}$ are constructible. We start be defining the Cayley-Menger determinant for  $\{\bfp_0,\bfp_1,\dots ,\bfp_k\}$, or equivalently for $\mathcal{D}=\{d_{ij}\} =\{\|\bfp_i-\bfp_j \|\}$ (see for instance \cite{Blumenthal:1943vx,MR839626}):
\begin{equation*}
\CM(\{\bfp_0,\bfp_1,\dots ,\bfp_k\}) = \CM(\mathcal{D}) = 
\left|
\begin{matrix}
0 & 1 & 1 & \dots & 1 \\
1 & 0 & d_{01}^2 & \dots & d_{0k}^2 \\
1 & d_{10}^2 & 0 & \dots & d_{1k}^2 \\
\vdots & \vdots & \vdots &  & \vdots \\
1 & d_{k0}^2 & d_{k1}^2 & \dots & 0 \\
\end{matrix}
\right| .
\end{equation*}

\begin{theorem}[\cite{MR882541} Theorem 9.7.3.4, and 9.14.23]
Given a set of nonzero distances $\mathcal{D}=\{d_{ij}\}$  for $i,j = 0,1,\dots, k$,  a necessary and sufficient condition for the existence of a simplex $\{\bfp_0,\dots, \bfp_k\}$ with $d_{ij}=\| \bfp_i-\bfp_j\|$ is that for every $h=2,\dots,k$, and every $h$-element subset of $\{0,1,\dots, k\}$ the corresponding Cayley-Menger determinant be nonzero and its sign be $(-1)^{h}$. %Note corrected an error in Berger, see problem 9.14.23.

Moreover, when  $\mathcal{D}=\{d_{ij}\} = \{\|\bfp_i - \bfp_j\|\}$ for $\bfp_0,\dots,\bfp_{k}\in\R^k$, the volume $V$ of the simplex with vertices $\bfp_0, \dots, \bfp_{k}$ obeys 
\begin{equation*}
\Vol^2(\mathcal{D}) = \frac{(-1)^{k+1}}{2^k(k!)^2} \CM(\mathcal{D}).
\end{equation*}
\label{thm:CM}
\end{theorem}

Recall that by definition, a simplex is non-degenerate.  This means that both the volume of the simplex and Cayley-Menger determinant are nonzero.

The Cayley-Menger determinant generalizes standard facts in triangle geometry: for instance, for a triangle with side lengths $a$, $b$, and $c$ we can write this determinant explicitly as 
\begin{equation*}
\CM(\{a,b,c\}) = a^4-2 a^2 b^2-2 a^2 c^2+b^4-2 b^2 c^2+c^4 = -(a+b+c)(a+b-c)(a-b+c)(-a+b+c).
\end{equation*}
and conclude that 
\begin{equation*}
\operatorname{Area}(\{a,b,c\})^2 = \frac{1}{16} (a+b+c)(a+b-c)(a-b+c)(-a+b+c).
\end{equation*}
This is Heron's formula for the area of the triangle. We can see the triangle inequality, (a criteria for constructability of a triangle), in these formulas: Theorem~\ref{thm:CM} holds if and only if one of the side lengths is greater than the sum of the other two. 

%%%%%%%%%%%%%%
\subsection{Simplices and configuration spaces}

We will now shift our viewpoint and,  with an abuse of notation, view the simplex $\Delta$ in $\R^k$ as an ordered $(k+1)$-tuple of points $\Delta=(\bfp_0, \dots, \bfp_k)$. Thus the simplex is  a point in the open configuration space $C_{k+1}(\R^k)$. We now wish to understand the set of points $\{(\bfq_0, \dots, \bfq_k)\}$ in $C_{k+1}(\R^k)$ that correspond to simplices similar to our given simplex $\Delta$. By similar, we mean there is a translation, rotation, and nonzero scaling of $\R^k$ which maps $(\bfq_0, \dots, \bfq_k)$ to $\Delta$.

\begin{definition} For a given non-degenerate simplex $\Delta=(\bfp_0,\dots, \bfp_k)$, we let $\Sim(\Delta)\subset C_{k+1}(\R^k)$ denote the space of simplices similar to $\Delta$. Then $\Sim(\Delta)$ is the set of all $\q\in C_{k+1}(\R^k)$ such that  
$$\rijl(\q)= r_{ijl}(\Delta) \quad \text{for all}\ \  i\neq j\neq l\neq i.$$
\end{definition}

That is, when $\q=(\bfq_0,\dots, \bfq_k)$ is extrinsically similar to $\Delta$, then we have
$$r_{ijl}(\q) = \frac{\|\bfq_i - \bfq_j\|}{\|\bfq_i - \bfq_l\|} = \frac{\|\bfp_i - \bfp_j\|}{\|\bfp_i - \bfp_l\|} = r_{ijl}(\Delta).$$ 

Our next aim is to show that $\Sim(\Delta)$ is a submanifold of $C_{k+1}(\R^k)$, by proving that  $\Sim(\Delta)$ is diffeomorphic to $\Or(k)\times (0,\infty)\times \R^k$. We also aim to understand how the boundary of $\Sim(\Delta)$ sits inside of $C_{k+1}[\R^k]$.  In order to do this, we will associate a matrix to $\Sim(\Delta)$, and look at the polar decomposition of that matrix. So we will need to use some results from linear algebra along the way. These are all found in Appendix~\ref{append:lin-alg}.

\begin{definition} Given a configuration $\q=(\bfq_0,\dots, \bfq_k)$ in $\Sim(\Delta)$, we define the $k\times k$ matrix $\Pi(\q)$ by 
$$\Pi(\q) = \begin{bmatrix} \pi_{10}(\q) & \dots & \pi_{k0}(\q) \end{bmatrix}.$$
\end{definition}

\begin{proposition}
Given a configuration $\q = \{\bfq_0, \ldots, \bfq_k\}$ in $\Sim(\Delta)$, the matrix $\Pi(\q)$ has rank $k$, and the $k\times k$ matrix $\Pi(\q)^T\Pi(\q)$ is given by
$$\Pi(\q)^T\Pi(\q) = \begin{bmatrix}\cos(\theta_{ij})\end{bmatrix} = \begin{bmatrix}\frac12 (r_{0ij}(\q) + r_{0ji}(\q) - r_{ij0}(\q) r_{ji0})(\q)\end{bmatrix} = P^2(\Delta),$$
where $P(\Delta)$ is a uniquely determined symmetric positive-definite matrix.
\label{prop:P-matrix}
\end{proposition}

The matrix $\Pi(\q)^T\Pi(\q)$ is called the Gram matrix and it consists of the dot products of the $\pi_{ij}(\q)$, which are the cosines of the angles between the unit vectors $\pi_{ij}(\q)$. 

\begin{proof} Recall the law of cosines: $c^2 = a^2 + b^2 - 2ab\cos C$, where $c$ is the length of the side of a triangle $\triangle ABC$ opposite angle $C$ and $a$ and $b$ are the lengths of the sides subtending angle $C$. In the case $a = \| \bfq_i - \bfq_0\|$, $b = \| \bfq_j - \bfq_0\|$, and $c = \| \bfq_i - \bfq_j\|$, angle $C$ is $\theta_{ij}$ and we have
\begin{align*}
\cos \theta_{ij} &= \dfrac{a^2 + b^2 - c^2}{2ab} = \dfrac12 \left( \dfrac{a}{b} + \dfrac{b}{a} - \dfrac{c^2}{ab} \right) \\
&= \dfrac12 (r_{0ij}(\q) + r_{0ji}(\q) - r_{ij0}(\q) r_{ji0}(\q)).
\end{align*}
Because the $r_{ijl}$ are the same for all configurations in $\Sim(\Delta)$, this matrix is the same as for $\Pi(\Delta)$ associated with $\Delta=(\bfp_0, \dots, \bfp_k)\in C_{k+1}(\R^k)$. The following shows $\Pi(\Delta)$ has rank $k$:
\begin{align*}
0 < \mbox{\rm Vol}_k(\Delta) &= \dfrac1{k!} \det\begin{bmatrix} \bfp_1 - \bfp_0 & \dots & \bfp_k - \bfp_0\end{bmatrix}\\
&= \dfrac{\| \bfp_1 - \bfp_0\|\times \cdots \times \| \bfp_k - \bfp_0\|}{k!} \det \begin{bmatrix}\pi_{10}(\p) &\dots & \pi_{k0}(\p)\end{bmatrix} \\
&=  \dfrac{\| \bfp_1 - \bfp_0\|\times \cdots \times \| \bfp_k - \bfp_0\|}{k!} \det\Pi(\Delta).
\end{align*} 

By construction, $\Pi(\Delta)^T\Pi(\Delta)$ is a symmetric matrix. We have just shown that $\Pi(\Delta)$ has full column rank. Thus by Theorem~\ref{HJ-727}, we know $\Pi(\Delta)^T\Pi(\Delta)$ is positive-definite.  Since $\Pi(\Delta)^T\Pi(\Delta)=\Pi(\q)^T\Pi(\q)$, we again use Theorem~\ref{HJ-727} to deduce $\Pi(\q)$ also has rank $k$.  Now, using other standard results from linear algebra (Theorem~\ref{HJ-726} and Remark~\ref{rmk-726}), we can deduce there is a unique symmetric positive-definite matrix $P(\Delta)$ such that $P(\Delta) = (\Pi(\Delta)^T \Pi(\Delta))^{1/2}$.  
\end{proof}

We now use the polar decomposition theorem for a matrix to get a better understanding of $\Sim(\Delta)$. Recall, from Theorem~\ref{thm:polar}, that the polar decomposition of an $k\times k$ real matrix $A$ is a factorization of the form $A=UP$, where $U$ is orthogonal and $P$ is positive semidefinite symmetric matrix. This decomposition is unique when $A$ is nonsingular. (Intuitively, if $A$ is interpreted as a linear transformation of $\R^k$, then the polar decomposition separates it into a rotation or reflection $U$ of $\R^k$, and a scaling of the space along a set of $k$ orthogonal axes.)

\begin{proposition}
If $\q \in \Sim(\Delta)$, then $\Pi(\q) = U(\q) P(\Delta)$ where $U(\q)$ is a uniquely determined $k\times k$ orthogonal matrix which is smooth function of $\q$. The $k\times k$ matrix $P(\Delta)$ depends only on $\Delta$, and is a symmetric positive-definite matrix. 
\label{prop:Sim-polar}
\end{proposition}

\begin{proof}
In our case, since $\rank\Pi(\q)=k$, the polar decomposition theorem (Theorem~\ref{thm:polar}) tells us that for the $k\times k$ matrix $\Pi(\q)$ there is a unique decomposition into $k\times k$ matrices: $\Pi(\q) = U(\q)P$. Here, matrix $U(\q)$ is orthogonal, and $P=(\Pi(\q)^T \Pi(\q))^{1/2}$. From Proposition~\ref{prop:P-matrix}, we know 
$$P=(\Pi(\q)^T \Pi(\q))^{1/2} = (\Pi(\Delta)^T \Pi(\Delta))^{1/2} = P(\Delta).$$
Thus, for each $\q\in \Sim(\Delta)$, we have the same symmetric positive-definite $P(\Delta)$ matrix. The dependence of $\Pi(\q)$ on $\q$ is clearly smooth; that $U(\q)$ depends smoothly on $\Pi(\q)$ and hence on $\q$ and $U(\q)$ is smooth were shown in \cite{Dieci-Eirola} (see Remark~\ref{rmk:P-smooth}).
\end{proof}

Let $\Or(k)$ be the set of all orthogonal $k\times k$ matrices. Roughly speaking, Proposition~\ref{prop:Sim-polar} says that for a (non-degenerate) simplex $\Delta$, we can obtain a different configuration in $\Sim(\Delta)$ by multiplying $P(\Delta)$ on the left by a matrix in $\Or(k)$.
This  leads us to define the following map.

\begin{definition} The {\em pose} map $ps: \Sim(\Delta) \rightarrow \Or(k)$ is defined by $ps(\q)=U(\q)$, where $U(\q)$ is the unique orthogonal matrix $U(\q)$ such that $\Pi(\q) = U(\q)P(\Delta)$ (as found in Proposition~\ref{prop:Sim-polar}).
For a simplex $\q\in \Sim(\Delta)$, we call the corresponding element $U(\q) \in \Or(k)$ the {\em pose} of $\q$.
\label{def:pose}
\end{definition} 

We can now deduce the structure of $\Sim(\Delta)$ as a submanifold of $C_{k+1}[\R^k]$.

\begin{theorem}
If $\Delta$ is a non-degenerate $k$-simplex in $\R^k$, then $\Sim(\Delta)\subset C_{k+1}(\R^k)$ is a submanifold diffeomorphic to $\Or(k) \times (0, \infty) \times \R^k$. 
\label{thm:SimD}
\end{theorem}

\begin{proof}
We will define the following pair of maps:
$$\Or(k) \times (0,\infty) \times \R^k \mapright{i_\Delta}{} \Sim(\Delta)  \subset C_{k+1}(\R^k) \mapright{ps_\Delta}{} \Or(k) \times (0,\infty) \times \R^k.$$
We will then prove that both maps are smooth, that the first map $i_\Delta$ is onto $\Sim(\Delta)$, and the composition is the identity. This will prove that $i_\Delta$ is a diffeomorphism onto $\Sim(\Delta)$. 

We have been given $\Delta=(\bfp_0,\ldots, \bfp_k)\in C_{k+1}(\R^k)$. We can then define the set of ratios $\{ r_{ijl}(\Delta)\}$, the set of unit vectors $\{\pij(\Delta)\}$, and the matrix $\Pi(\Delta)$.  We know that $\Pi(\Delta) = U(\Delta)P(\Delta)$, from Proposition~\ref{prop:Sim-polar}. Moreover, without loss of generality, we can assume that $\Delta$ is such that $U(\Delta)=I_k$ (the identity matrix), and so $\Pi(\Delta) = P(\Delta)$.

Let $A=\begin{bmatrix} \bfv_1 \ \dots \ \bfv_k\end{bmatrix}\in \Or(k)$, and define a new set of  unit vectors by $\pi_{ij} =\begin{bmatrix} \bfv_1 \ \dots \ \bfv_k\end{bmatrix} \pi_{ij}(\Delta)$. (Intuitively, the new $\pij$ are $\pij(\Delta)$ rotated/reflected by $A$.) If we also let $r_{ijl} = r_{ijl}(\Delta)$, then we define the map $i_\Delta$ by 
\begin{align*} i_\Delta (A, \lambda, \bfq_0) & = (\bfq_0, \bfq_0 + \lambda \pi_{10}, \bfq_0 + \lambda r_{021} \pi_{20}, \ldots, \bfq_0 + \lambda r_{0k1} \pi_{k0}) = \q.
 \end{align*}
This map $i_\Delta$ is clearly smooth in $A$, $\lambda$, and $\bfq_0$. Furthermore, $i_\Delta(A, \lambda, \bfq_0)$ has $r_{ijl}(\q)=r_{ijl}(\Delta)$ by construction, and so  lies in $\Sim(\Delta)$. 

We prove that $i_\Delta$ is onto $\Sim(\Delta)$. Given $\q\in \Sim(\Delta)$, we use Proposition~\ref{prop:Sim-polar} and our assumption that $U(\Delta)=I_k$, to write
$$\Pi(\q) = \begin{bmatrix} \pi_{10}(\q)\ \dots \ \pi_{k0}(\q)\end{bmatrix} = U(\q) P(\Delta)=U(\q)\Pi(\Delta).$$
This means that $\pi_{j0}(\q) = U(\q)\pi_{j0}(\Delta)$ for $j=1,\dots, k$. Remembering that $\pi_{j0}=-\pi_{0j}$, we then deduce the other values of $\pi_{j}(\q)$ satisfy the same relationship:
\begin{align*}
\pi_{ij}(\q) &= \dfrac{\bfq_j - \bfq_i}{\|\bfq_j - \bfq_i\|} = \dfrac{\bfq_j - \bfq_0}{\|\bfq_j - \bfq_i\|} + \dfrac{\bfq_0 - \bfq_i}{\|\bfq_j - \bfq_i\|} \\
&= \dfrac{\| \bfq_j - \bfq_0\|}{\| \bfq_j - \bfq_i\|} \dfrac{\bfq_j - \bfq_0}{\|\bfq_j - \bfq_0\|} +  \dfrac{\|\bfq_0 - \bfq_i\|} {\| \bfq_j - \bfq_i\|} \dfrac{\bfq_0 - \bfq_i}{\|\bfq_0 - \bfq_i\|}  \\
& = r_{j0i}(\q) \pi_{j0}(\q) + r_{i0j}(\q) \pi_{0i}(\q) \\
& = r_{j0i}(\q) U(\q)\pi_{j0}(\Delta) + r_{i0j}(\q) U(\q)\pi_{0i}(\Delta) = U(\q)\pi_{ij}(\Delta). 
\end{align*}
Thus for $\q\in \Sim(\Delta)$, we have $\pi_{ij}(\q) = U(\q)\pi_{ij}(\Delta)$. For the map $i_\Delta$, we let $\lambda=\|\bfq_1-\bfq_0\|$, then $i_\Delta( U(\q), \|\bfq_1-\bfq_0\|, \bfq_0) = (\bfq_0,\bfq_1,\dots, \bfq_k)=\q$. To see this last equation, note that $$\bfq_1=\bfq_0+\lambda\frac{\bfq_1-\bfq_0}{\|\bfq_1-\bfq_0\|} \quad \text{and}\quad \bfq_2= \bfq_0 + \lambda r_{021}(\q) \pi_{20}(\q) = \bfq_0+\lambda\dfrac{\| \bfq_2-\bfq_0 \|}{\| \bfq_1-\bfq_0 \|}\dfrac{\bfq_2-\bfq_0}{\| \bfq_2-\bfq_0 \|}.$$

We can now define the map $ps_\Delta\co \Sim(\Delta)\rightarrow \Or(k)\times(0,\infty)\times \R^k$ using the pose map from Definition~\ref{def:pose}. We define  $ps_\Delta(\q) := (ps(\q),\|\bfq_1-\bfq_0\|, \bfq_0) = (U(\q),\|\bfq_1-\bfq_0\|, \bfq_0)  $. From Proposition~\ref{prop:Sim-polar}, we know $U(\q)$ depends smoothly on $\q$,  and hence $ps_\Delta$ depends smoothly on $\q$.  A moment's thought shows that, by construction, we have $ps_\Delta\circ i_\Delta =\id$. 

This means that $i_\Delta$ (and hence $ps_\Delta$) is a diffeomorphism, and $\Sim(\Delta)$ is a submanifold of $C_{k+1}(\R^k)$.
\end{proof}

We note that this theorem shows that the polar decomposition is a diffeomorphism for matrices of full rank. This is analogous to the smoothness result for the polar decomposition of Dieci and Eirola~\cite{Dieci-Eirola}. This theorem is not true for rank-deficient matrices, which explains why we have restricted our attention to spheres $S^{k-1}$ isotopic to each other in $\R^k$; Haefliger's theorem might guarantee the existence of an isotopy of the spheres in a higher-dimensional $\R^K$, but $\Sim(\Delta)$ would still consist of rank $k < K$ matrices and hence be harder to control.

The next theorem shows $\Sim(\Delta)$ has a well understood structure in the boundary $\bdry C_{k+1}[\R^k]$.

\begin{theorem}
If $\Delta$ is a non-degenerate $k$-simplex in $\R^k$, then the boundary of $\Sim(\Delta)$ corresponds to configurations in the interior of the $(0, \ldots, k)$ face of $\partial C_{k+1}[\R^k]$, and is diffeomorphic to $\Or(k) \times \{0\} \times \R^k$.
\label{thm:SimD-bdry}
\end{theorem}

\begin{proof} By assumption the given simplex $\Delta=(\bfp_0, \dots, \bfp_k)$ is non-degenerate, so none of the vertices of $\Delta$ coincide, and the ratios $\rijl(\Delta)$ are never 0 nor $\infty$.  If we take $(A,\lambda,\bfp)\in \Or(k)\times [0,\infty)\times\R^k$ and consider the points in $i_\Delta(A,\lambda, \bfp_0)$, we see these points coincide if and only if $\lambda=0$, in which case they all coincide. This means $i_\Delta(A,0, \bfp_0)\subset\bdry C_{k+}[\R^k]$ and is in the $(0,1,\dots, k)$ face. Since the boundary of the  $(0,1,\dots, k)$ face consists of configurations with $\rijl(\Delta) = 0$, or $\rijl(\Delta) = \infty$, it follows immediately that $i_\Delta(A,0, \bfp_0)$ is in the interior of this face.
\end{proof}

\begin{corollary} 
If $\Delta$ is a non-degenerate $k$-simplex in $\R^k$, then $\Sim(\Delta)$ is a submanifold of $C_{k+1}[\R^k]$ that is diffeomorphic to $\Or(k) \times [0,\infty) \times \R^k$.
\label{cor:SimD}
\end{corollary}

%%%%%%%%%%%%%%%%%%%%%%
%%%%%%%%%%%%%%%%%%%%%%%%%%

\section{Simplices inscribed in spheres}
\label{sect:inscribed}

We next move to Step 1 of the method described in Section~\ref{sect:main-method}.  We will be looking at $C^\infty$-smooth embeddings of  $S^{k-1}$ in $\R^k$. We will always assume that these embeddings are {\em regular}, that is, the embedding induces an injection of tangent spaces everywhere. This is particularly relevant for embeddings of $S^1$ in $\R^2$ where we assume the tangent vector is nowhere zero. (Otherwise it is possible to smoothly describe an embedded curve with corners, by allowing the tangent vector to smoothly change to zero at each corner.)  Now given any such regular $C^\infty$-smooth embedding $\gamma\co S^{k-1}\hookrightarrow \R^k$, we can view the corresponding configuration space $C_{k+1}[\gamma(S^{k-1})]$ as a submanifold of $C_{k+1}[\R^k]$.

\begin{proposition} \label{prop:bdry-disjoint}
Given a non-degenerate simplex $\Delta\in C_{k+1}(\R^k)$, and a regular $C^\infty$-smooth embedding $\gamma\co S^{k-1}\hookrightarrow \R^k$ with corresponding configuration space $C_{k+1}[\gamma(S^{k-1})]$, then $\bdry \Sim(\Delta)$ and $\bdry C_{k+1}[\gamma(S^{k-1})]$ are disjoint in $\bdry C_{k+1}[\R^k]$.
\end{proposition}

\begin{proof} From Theorem~\ref{thm:SimD-bdry} we know that $\bdry \Sim(\Delta)$ is in the $(0,1,\dots, k)$ face of $\bdry C_{k+1}[\R^k]$, so we restrict our attention to that boundary face. Since $\Delta$ is non-degenerate, then no hyperplane in $\R^k$ contains more than $k$ points of $\Delta$, and so not all of the $\pij(\Delta)$ vectors are coplanar.  However, for a point in $\bdry C_{k+1}[\gamma(S^{k-1})]$ all of the $\pij$ vectors must be coplanar. Hence $\Sim (\Delta)$ and $C_{k+1}[\gamma(S^{k-1})]$ are boundary disjoint.
\end{proof}

We next move to Step 2.   We take the standard embedding of the $(k-1)$ sphere in $\R^k$, that is $\id \co S^{k-1} \hookrightarrow \R^k$, and consider the corresponding configuration spaces. In this special case, we use the notation $C_{k+1}[S^{k-1}] = C_{k+1}[\id(S^{k-1})]$.

We need to show that there is a transverse intersection between $C_{k+1}[S^{k-1}]$ and $\Sim(\Delta)$ in $C_{k+1}[\R^k]$, in other words $C_{k+1}(\id)\transverse \Sim(\Delta)$. We also need to compute the homology class of the intersection of $C_{k+1}[S^{k-1}]\cap \Sim(\Delta)$  in $\Sim(\Delta)$.

\begin{proposition}
Given the configuration space $C_{k+1}[S^{k-1}]$ corresponding to the standard embedding of $S^{k-1}$ in $\R^k$, and given a non-degenerate simplex $\Delta \in C_{k+1}(\R^k)$, then $\Sim(\Delta)$ intersects $C_{k+1}[S^{k-1}]$ transversally, and the intersection $\Sim(\Delta)\cap C_{k+1}[S^{k-1}]$ is diffeomorphic to $\Or(k)$.
\label{prop:transverse}
\end{proposition}

\begin{proof} Since $\Sim(\Delta)$ and $C_{k+1}[S^{k-1}]$ are boundary disjoint (Proposition~\ref{prop:bdry-disjoint}), we just need to consider $\Sim(\Delta) \cap C_{k+1}(S^{k-1})$. Now, every simplex $\Delta$ has a unique circumsphere; a $(k-1)$-sphere passing through all of the $k+1$ vertices. The radius of the circumsphere and coordinates of the circumcenter are well known (see for instance Proposition 9.7.3.7 \cite{MR882541}, or \cite{Coxeter, MR1503248}).  Indeed the circumradius $R$ of the simplex $\Delta$ is given by
\begin{equation*}
R^2 = - \frac{
\left|
\begin{matrix}
0 & d_{01}^2 & \cdots & d_{0k}^2 \\
d_{10}^2 & 0 & \cdots & d_{1k}^2 \\
\vdots & \vdots & & \vdots \\
d_{k0}^2 & d_{k2}^2 & \cdots & 0 
\end{matrix}
\right| 
} 
{2 \CM(\{\bfp_0,\dots, \bfp_k\})} .
\end{equation*}

Given $\Delta$, we can scale and translate it to give a new $\q\in\Sim(\Delta)$ with circumradius $R=1$, and circumcenter $\bfzero$. Thus $\q\in \Sim(\Delta) \cap C_{k+1}(S^{k-1})$ and the intersection is nonempty.  Intuitively,  we obtain all other points of $\Sim(\Delta) \cap C_{k+1}(S^{k-1})$ by rotating/reflecting $\q$ via $A\q$ for $A\in \Or(k)$.   More formally, we define the diffeomorphism $r\co \Or(k)\rightarrow  \Sim(\Delta) \cap C_{k+1}(S^{k-1})$ by $r(A) =A\q = (A\bfq_0, \dots , A\bfq_k)$. 

We now want to show that $\Sim(\Delta)$ intersects $C_{k+1}[S^{k-1}]$ transversally. Specifically, we take any inscribed simplex $\q\in \Sim(\Delta)\cap C_{k+1}[S^{k-1}]$, and we want to show that
$$ T_{\q} (\Sim(\Delta)) \oplus T_{\q} (C_{k+1}[S^{k-1}]) = T_{\q}(C_{k+1}[\R^k]).$$
We first note that the orthogonal complement of $T_{\q} (C_{k+1}[S^{k-1}])$ in $T_{\q}(C_{k+1}[\R^{k}])$ is the $(k+1)$-dimensional space with orthonormal basis $\mathcal{B} = \{(\bfq_0, 0, \dots, 0), \dots, (0,\dots,\bfq_{k})\}$. Next, observe that $T_{\q} (\Sim(\Delta)) \cong T_{\q}(\Or(k))\oplus T_{\q}[0,\infty) \oplus T_{\q}(\R^k)$. The tangent space $T_{\q}(\Sim(\Delta))$ thus contains the vectors $(\bfe_1, \dots, \bfe_1), \dots, (\bfe_k,\dots,\bfe_k)$ from the translational component of $\Sim(\Delta)$ as well as the vector $(\bfq_0,\dots,\bfq_{k})$ from scaling the configuration $\q$. Writing these vectors in the basis $\mathcal{B}$, we get the matrix:
\begin{equation*}
M = \left( 
\begin{matrix}
q_{0,0} & q_{1,0} & \cdots & q_{k,0} \\
q_{0,1} & q_{1,1} & \cdots & q_{k,1} \\
\vdots  & \vdots  &        &   \vdots  \\
q_{0,k} & q_{1,k} & \cdots & q_{k,k} \\
1       & 1       & \cdots & 1         \\
\end{matrix}
\right).
\end{equation*}
Subtracting the last column from the rest, we get 
\begin{equation*}
M' = \left( 
\begin{matrix}
\bfq_0 - \bfq_{k} & \bfq_1 - \bfq_{k} & \cdots & \bfq_{k-1} - \bfq_{k} & \bfq_{k} \\
0             & 0             & \cdots & 0             & 1       \\
\end{matrix}
\right).
\end{equation*}
The determinant of this matrix is $\pm 1$ multiplied by the determinant of the upper-left $k \cross k$ principal minor. But that determinant is positive because $\q \in \Sim(\Delta)$ and is non-degenerate. Thus the $k+1$ tangent vectors are linearly independent and we have proven transversality.
\end{proof}

Before we move on to determine the homology class corresponding to inscribed simplices, we pause to remember that $\Or(k)$ has two connected components. One component, $\SO(k)$, is a subgroup of $\Or(k)$, and consists of all orthogonal matrices with determinant $+1$. The other component contains all orthogonal matrices with determinant $-1$. 

In order to sensibly discuss homology classes, we restrict our attention to the submanifold of simplices diffeomorphic to $\SO(k)\times [0,\infty)\times \R^k$, which we denote $\Sim^+(\Delta)$. We could equally restrict our attention to $\Sim^-(\Delta) := \Sim(\Delta)\setminus\Sim^+(\Delta)$.

\begin{proposition} Given a non-degenerate simplex $\Delta=(\bfp_0,\dots, \bfp_k)\in C_{k+1}(\R^k)$, and given the configuration space $C_{k+1}[S^{k-1}]$ corresponding to the standard embedding of $S^{k-1}$ in $\R^k$, then in $\Sim(\Delta)$ 
$$H_*(\Or(k);\Z) \cong H_*(\Sim(\Delta) \cap C_{k+1}[S^{k-1}];\Z).$$

Moreover, the pose map is a diffeomorphism $ps\co \Sim^+(\Delta)\cap C_{k+1}[S^{k-1}] \rightarrow \SO(k)$
which induces an isomorphism on integral homology taking the top class of $\Sim^+(\Delta)\cap C_{k+1}[S^{k-1}]$ in $\Sim(\Delta^+)$ 
to the top class of $\SO(k)$. A similar result holds for $\Sim^-(\Delta)\cap C_{k+1}[S^{k-1}]$. 
\label{prop:homology}
\end{proposition}

\begin{proof}
Without loss of generality, we can normalize $\Delta=(\bfp_1,\dots, \bfp_k)$ so that it lies on $S^{k-1}$; that is $\Delta$ has circumcenter $\bfzero$ and circumradius $R=1$.  We then have the following maps
$$\Or(k)\xrightarrow{r} \Sim(\Delta)\cap C_{k+1}[S^{k-1}] \xrightarrow{ps} \Or(k).$$
Here, the map $r$ is the rotation/reflections diffeomorphism previously defined in Proposition~\ref{prop:transverse} That is, for $A\in \Or(k)$, we define $r(A)= A\Delta=(A\bfp_0, \dots , A\bfp_k)$. The map $ps$ is the pose map from Definition~\ref{def:pose}:  for $\q\in\Sim(\Delta)$, we define $ps(\q)=U(\q)$ (where $U(\q)$ is the unique orthogonal matrix such that $\Pi(\q) = U(\q)\Pi(\Delta)$). 
From Proposition~\ref{prop:transverse} we know that $r$ is a diffeomorphism, and by construction $ps\circ r = \id$. Hence the pose map is also a diffeomorphism. 

From Proposition~\ref{prop:transverse}, we recall that $\Sim(\Delta)$ intersects $C_{k+1}[S^{k-1}]$ transversally, and moreover that the intersection is diffeomorphic to $\Or(k)$. Now, since $\Sim(\Delta)\cong \Or(k)\times [0,\infty)\times \R^k$, we have $\Sim(\Delta)\cap C_{k+1}[S^{k-1}]$ is a deformation retract of $\Sim(\Delta)$.

When we put this together and take the homology of the spaces we see
$$H_*(\Or(k))\xrightarrow{r_*} H_*( \Sim(\Delta)\cap C_{k+1}[S^{k-1}]) \cong H_*(\Sim(\Delta)) \xrightarrow{ps_*} H_*(\Or(k))
$$
Since $ps_*  \circ r_* = \id$, we know $ps_*$ is an isomorphism. Hence 
$$H_*(\Or(k)) \cong H_*(\Sim(\Delta) \cap C_{k+1}[S^{k-1}]).$$

We could also choose to restrict $r$ and $ps$ as follows: 
$$\SO(k)\xrightarrow{r} \Sim^+(\Delta)\cap C_{k+1}(S^{k-1}) \xrightarrow{ps} \SO(k).$$
Then, repeating the previous argument gives 
$$ps_*([\Sim^+(\Delta) \cap C_{k+1}(S^{k-1})] )=[r(\SO(k))] = [\Sim^+(\Delta)] \in H_{\frac{k(k-1)}{2}}(\Sim^+(\Delta)).$$
\end{proof}

We now have all the pieces needed for our main theorem. 

\begin{theorem} \label{thm:main-theorem}
Suppose $\gamma\co S^{k-1}\hookrightarrow \R^k$ is a $C^\infty$-smooth embedding of $S^{k-1}$ in $\R^k$ isotopic to the identity through a differentiable isotopy in $\R^k$, with a corresponding embedding of compactified configuration spaces $C_{k+1}[\gamma]\co C_{k+1}[S^{k-1}]\hookrightarrow C_{k+1}[\R^k]$. Assume that $\Delta\in C_{k+1}(\R^k)$ is a non-degenerate simplex.

Then for all $\epsilon >0$, there is a  $C^\infty$-open neighborhood of $\gamma$, in which there is, for all $m$, a $C^m$-dense set of smooth embeddings $\gamma':S^{k-1}\hookrightarrow \R^k$ isotopic to the identity  through a differentiable isotopy in $\R^k$, such that $\| C_{k+1}[\gamma'] - C_{k+1}[\gamma]\|_0< \epsilon$, and $C_{k+1}[\gamma'] \transverse \Sim(\Delta)$. 
Moreover in $\Sim(\Delta^+)$, both $C_{k+1}[S^{k-1}] \cap \Sim^+(\Delta)$ and $C_{k+1}[\gamma'(S^{k-1})] \cap \Sim^+(\Delta)$ represent the top homology class of $\SO(k)$.   A similar result holds for $\Sim^-(\Delta)$.
\end{theorem}

\begin{proof} This is Corollary~\ref{cor:SimD}, and Propositions~\ref{prop:transverse} and~\ref{prop:homology} applied to Theorem~\ref{thm:method}.
\end{proof}

We immediately recover Gromov's result~\cite{MR244907} (but for $C^\infty$-smooth embeddings) as a corollary. 

\begin{corollary} For any non-degenerate $k$-simplex $\Delta$ in $\R^k$, there is a dense family of smoothly embedded $(k-1)$-spheres in $\R^k$ isotopic to the identity through a differentiable isotopy in $\R^k$, such that the subset of $\Sim(\Delta)$ of simplices inscribed in each embedded sphere contains a similar simplex corresponding to each $U\in \Or(k)$. 
\label{cor1}
\end{corollary}

In fact we have proved more than Gromov.

\begin{corollary}  For any non-degenerate $k$-simplex $\Delta$ in $\R^k$, there is a dense family of smoothly embedded $(k-1)$-spheres in $\R^k$ isotopic to the identity through a differentiable isotopy in $\R^k$, 
such that the subset $\Sim^+(\Delta)\cap C_{k+1}[\gamma(S^{k-1})]$ of simplices inscribed in each embedded sphere is a smooth orientable submanifold of $\Sim^+(\Delta)$. Furthermore, the pose map $ps\co  \Sim^+(\Delta)\cap C_{k+1}[\gamma(S^{k-1})]\rightarrow \SO(k)$ is onto and is a degree 1 map. 

In particular, the set of triangles similar to a given triangle $\triangle ABC$ inscribed in a generic smooth plane curve is a collection of loops $L_1,\dots, L_n$. Furthermore, the sum of the degrees of the pose maps $ps_1,\dots, ps_n$ is one.
\label{cor2}
\end{corollary}

\begin{proof} The first statement follows from Theorem~\ref{thm:homology}.  That the pose map is onto follows from Corollary~\ref{cor1}.  For the standard embedding of $S^{k-1}$ in $\R^k$, the pose map is a diffeomorphism, and hence of degree 1. Proposition~\ref{prop:homology} and homotopy invariance show the pose map is always of degree 1 for our embeddings.

When $k=2$, we know $\SO(2)\cong S^1$. Hence the set of inscribed triangles similar to $\Delta$ is a collection of loops.
\end{proof}

\begin{figure}
\includegraphics[width=5.5in]{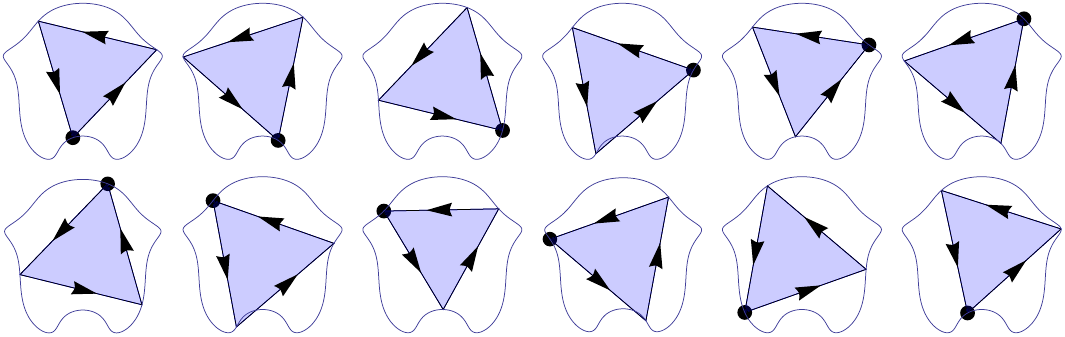}
\hphantom{.}
\caption{In this irregular embedding of a circle in the plane, we see a series of inscribed equilateral triangles. By following the highlighted vertex around the curve, we see there is a loop of such inscribed triangles.}
\label{fig:triangle-loop}
\end{figure}

In \cite{phd-Matschke}, Matschke proves that for generic $C^\infty$-smooth embeddings of the circle in the plane, there are an odd number of loops of inscribed triangles similar to $\Delta$ that wind an odd number of times around the embedded circle. Figure~\ref{fig:triangle-loop} shows such a loop of inscribed equilateral triangles in an irregular embedding of a circle in the plane.
We recover this Matschke's result in Corollary~\ref{cor2}, and strengthen it with the fact that the degree of the map is 1. 

We end by again noting that Meyerson~\cite{MR600575} and Nielsen~\cite{MR1181760}  have results about inscribed triangles that are for Jordan curves (just assuming continuity of the embedding). By adding both a generic and smoothness assumption on our embeddings, we are able to provide more information about the structure of the set of inscribed triangles.

%%%%%%%%%%%%%%%%%%%%%%%%%%%%%%%%%
%%%%%%%%%%%%%%%%%%%%%%%%%%%%%%%%

\section*{Acknowledgements}
We would like to thank all the people who have discussed these problems with us over the years, especially Jordan Ellenberg, Benjamin Matschke,  and  Gunter Ziegler.
%%%%%%%%%%%%%%%%%%%%%%%%%%%%%%

\bibliography{simplices}{}
\bibliographystyle{plain}

%%%%%%%%%%%%%%%%%%%%%%%%%
%%%%%%%%%%%%%%%%%%%%%%%%

\appendix

\section{Results from Linear Algebra}\label{append:lin-alg}

We will  let $M_n$ denote the set of all $n\times n$ matrices, and let $M_{n,m}$ denote the set of all $n\times m$ matrices. Recall that a symmetric matrix is a square matrix that is equal to its transpose: $A=A^T$. That is, $A=[a_{ij}]$ is symmetric if and only if $a_{ij}=a_{ji}$.
%Recall that a Hermitian matrix is a complex square matrix that is equal to its conjugate transpose. That is $A=[a_{ij}]$ is Hermitian if and only if $a_{ij} = \overline{a_{ji}}$. In this case we write $A=A^*$. Hermitian matrices are the complex extension of real symmetric matrices. The matrices we will work with are all real. 
Also recall that 
\begin{compactitem}
\item An $n\times n$ symmetric real matrix $M$ is {\em positive-definite} if $\bfx^*M\bfx >0$ for all $\bfx \in \R^n\setminus\{\bfzero\}$. 
\item An $n\times n$ symmetric real matrix $M$ is {\em positive semidefinite} of {\em non-negative-definite} if $\bfx^*M\bfx \geq 0$ for all  $\bfx \in \R^n$. 
\item Similar definitions hold for negative-definite and negative semidefinite.
\end{compactitem}
%
%\item An $n\times n$ Hermitian complex matrix $M$ is {\em positive definite} if $\bfx^*M\bfx >0$ for all $\bfx \in \C^n\setminus\{\bfzero\}$. 
%\item An $n\times n$ Hermitian complex matrix $M$ is {\em positive semidefinite} of {\em non-negative definite} if $\bfx^*M\bfx \geq 0$ for all  $\bfx \in \C^n$. 
%\item Similar definitions hold for negative definite and negative semidefinite.

We now review some theorems about symmetric matrices.

\begin{theorem}[See \cite{Horn-Johnson} Theorem 7.2.6]
Let $A$ be an $n\times n$ symmetric and positive semidefinite matrix, let $r=\rank(A)$, and let $k=\{2, 3, \dots\}$. Then there is a unique symmetric positive semidefinite matrix $B$ such that $B^k=A$.
\label{HJ-726}
\end{theorem}

\begin{remark} While this theorem only guarantees a unique positive {\em semidefinite} square root for a  {\em semidefinite} matrix $A$, an examination of the proof in \cite{Horn-Johnson} shows that more has been proven. The proof shows there is a unique positive {\em definite} square root for a positive {\em definite} matrix $A$.
\label{rmk-726}
\end{remark}

\begin{theorem}[See \cite{Horn-Johnson} Theorem 7.2.7]
Let $A$ be an $n\times n$ symmetric matrix.
If $A=B^TB$, with $B$ an $m\times n$ matrix, then $A$ is positive definite if and only if $B$ has full column rank.
\label{HJ-727}
\end{theorem}

%\begin{theorem}[See \cite{Horn-Johnson} Theorem 7.2.7]
%Let $A$ be an $n\times n$ Hermitian matrix.
%\begin{compactenum}
%\item $A$ is positive semidefinite if and only if there is an $m\times n$ matrix $B$ such that $A=B^*B$.
%\item If $A=B^*B$, with $B$ an $m\times n$ matrix, and if $x\in \C^n$, then $Ax=0$ if and only if $Bx=0$. So $\nullspace A=\nullspace B$ and $\rank A = \rank B$.
%\item If $A=B^*B$, with $B$ an $m\times n$ matrix, then $A$ is positive definite if and only if $B$ has full column rank.
%\end{compactenum}
%\label{HJ-727}
%\end{theorem}

The polar decomposition of a matrix is incredibly useful. Below, we give the version that best applies to our work.

\begin{theorem}[See \cite{Horn-Johnson} Theorem 7.3.1 Polar decomposition]
Let $A$ be an $n\times n$ real matrix. Then $A=PU=UQ$, in which $P,Q\in M_n$ are positive semidefinite and $U\in M_n$ is orthogonal. The factors $P=(AA^T)^{1/2}$ and $Q=(A^TA)^{1/2}$ are uniquely determined; $P$ is a polynomial in $AA^T$ and $Q$ is a polynomial in $A^TA$. The factor $U$ is uniquely determined if $A$ is nonsingular.
\label{thm:polar}
\end{theorem}

%\begin{theorem}[See \cite{Horn-Johnson} Theorem 7.3.1 Polar decomposition]
%Let $A$ be an $n\times m$ matrix.
%\begin{compactenum}
%\item If $n<m$, then $A=PU$, in which $P\in M_n$ is positive semidefinite and $U\in M_{n,m}$ has orthornormal rows. The factor $P=(AA^*)^{1/2}$ is uniquely determined; it is a polynomial in $AA^*$. The factor $U$ is uniquely determined if $\rank A=n$.
%\item If $n=m$, then $A=PU=UQ$, in which $P,Q\in M_n$ are positive semidefinite and $U\in M_n$ is unitary. The factors $P=(AA^*)^{1/2}$ and $Q=(A^*A)^{1/2}$ are uniquely determined; $P$ is a polynomial in $AA^*$ and $Q$ is a polynomial in $A^*A$. The factor $U$ is uniquely determined if $A$ is nonsingular.
%\item If $n>m$, then $A=UQ$, in which $Q\in M_m$ is positive semidefinite and $U\in M_{n,m}$ has orthornormal columns. The factor $Q=(A^*A)^{1/2}$ is uniquely determined; it is a polynomial in $A^*A$. The factor $U$ is uniquely determined if $\rank A=m$.
%\item If $A$ is real, the factors $P$, $Q$ and $U$ in (1), (2) and (3) may be taken to be real.
%\end{compactenum}
%\end{theorem}

\begin{remark}[See \cite{Dieci-Eirola} Section 2.3 (c)] Suppose $A$ is an $m\times n$ matrix with full rank and $A(t)$ depends $C^k$-smoothly on $t$. Then we can write $A(t)=O(t)P(t)$, where $O$ is orthonormal, $P$ is symmetric positive definite and $O$ and $P$ are as smooth as $A$. 
%
%[To see this, we differentiate $A=OP$, use the orthonormality of $O$ and let $H=O^T\dot{O}$ to get
%\begin{align*}
%\dot{O} & = OH
%\\ \dot{P}&=O^T\dot{A} - HP
%\end{align*}
%Since $P=P^T$, we then must have
%\begin{equation}
%PH + HP = O^T\dot{A} - \dot{A}^T O
%\label{eq:2-20}
%\end{equation}
%Since $P$ is positive definite Equation~\ref{eq:2-20} has a unique solution $H$. This gives the desired result, since $H$ satisfying Equation~\ref{eq:2-20}  is skew-symmetric.]
\label{rmk:P-smooth}
\end{remark}

%%%%%%%%%%%%%%%%%%%%%%%%%
%%%%%%%%%%%%%%%%%%%%%%%%%

\end{document}